\documentclass[letterpaper,11pt]{article}
\usepackage[T1]{fontenc}
\usepackage{booktabs}
\usepackage{color}
\usepackage{enumitem}
\usepackage{algorithm}
\usepackage{algpseudocode}
\usepackage{mathtools}
\usepackage{url}
\usepackage{amssymb, amsmath}
\usepackage{amsthm}
\usepackage{amsfonts}
\usepackage{physics}
\usepackage{mathrsfs}

\newtheorem{theorem}{Theorem}[section]
\newtheorem{lemma}[theorem]{Lemma}

\theoremstyle{definition}
\newtheorem{definition}[theorem]{Definition}
\theoremstyle{remark}

\def\noteson{\gdef\ines##1{\noindent{\color{blue}[Inés: ##1]}}}
\noteson

\renewcommand{\H}{\mathrm{H}}

\newcommand{\rk}{\mathrm{rk}}

\DeclareMathAlphabet{\mathpzc}{OT1}{pzc}{m}{it}
\usepackage[mathscr]{euscript}

\makeatletter
\newcommand\DEFINEALPHABETLOOP[3]{%
  \ifx\relax#3\expandafter\@gobble\else\expandafter\@firstofone\fi
  {\expandafter\newcommand\expandafter*\csname#3#1\endcsname{#2{#3}}%
   \DEFINEALPHABETLOOP{#1}{#2}}%
}%
\newcommand\Definealphabet[2]{%
  \DEFINEALPHABETLOOP{#1}{#2}abcdefghijklmnopqrstuvwxyzABCDEFGHIJKLMNOPQRSTUVWXYZ\relax
}%
\makeatother
\Definealphabet{bb}{\mathbb}
\Definealphabet{cal}{\mathcal}
\Definealphabet{bf}{\mathbf}
\Definealphabet{sf}{\mathsf}
\Definealphabet{rm}{\mathrm}
\Definealphabet{frak}{\mathfrak}
\Definealphabet{scr}{\mathscr}
\begin{document}
\title{Confidence Bands for\\Multiparameter Persistence Landscapes}

\author{Inés García-Redondo\thanks{LSGNT, Imperial College, London, UK; \texttt{i.garcia-redondo22@imperial.ac.uk}} \and Anthea Monod\thanks{Imperial College, London, UK; \texttt{a.monod@imperial.ac.uk}}
\and Qiquan Wang\thanks{Imperial College, London, UK; \texttt{qiquan.wang17@imperial.ac.uk}}}

\maketitle            
\begin{abstract}

Multiparameter persistent homology is a generalization of classical persistent homology, a central and widely-used methodology from topological data analysis, which takes into account density estimation and is an effective tool for data analysis in the presence of noise. Similar to its classical single-parameter counterpart, however, it is challenging to compute and use in practice due to its complex algebraic construction. In this paper, we study a popular and tractable invariant for multiparameter persistent homology in a statistical setting: the multiparameter persistence landscape. We derive a functional central limit theorem for multiparameter persistence landscapes, from which we compute confidence bands, giving rise to one of the first statistical inference methodologies for multiparameter persistence landscapes. We provide an implementation of confidence bands and demonstrate their application in a machine learning task on synthetic data.

\end{abstract}
\section{Introduction}

\emph{Persistent homology} (PH) is a methodology from topological data analysis (TDA) that has enjoyed widespread success in diverse application areas in recent decades as an effective tool for data analysis, statistics, and machine learning. It provides a robust framework and computationally efficient algorithms for extracting topological descriptors from data, aiming to capture their underlying shape and structure. The core idea is to construct a sequence of nested topological spaces \(\{K_t: t \in T\}\), known as a \emph{filtration}, typically indexed by an interval of the real line \(T \subset \mathbb{R}\), and such that $K_t \subset K_s$ whenever $t \leq s$. Topological invariants that are amenable to data analysis are computed by tracking how topological features, characterized by homology, appear and disappear as the filtration parameter evolves. This procedure processes data and outputs a \textit{persistence diagram}, a statistical summary capturing the \textit{lifetimes} of its topological features. The statistical and probabilistic properties of persistence diagrams, as random objects, have been a rich and challenging area of study within TDA, since the algebraic construction of persistence diagrams induces an intricate, non-Euclidean geometry in the space of all persistence diagrams. A common workaround to this complex geometry is to vectorize or embed the persistence diagrams into a Euclidean or other space with a more standard geometry for statistics and machine learning methods; one of the most popular vectorizations is known as the \emph{persistence landscape} \cite{Bubenik2015}.

The classical setting of \emph{single-parameter} PH described above can be extended to \emph{multiparameter persistent homology} (MPH), where \(T\) becomes a multidimensional set, often \(\mathbb{R}^d\) with the product order. Practically, this extension incorporates additional structural information---such as density estimation---into the filtration and is an effective adaptation to the classical setting in the presence of noise. MPH is significantly more challenging to compute and use in real data settings, as it lacks a natural invariant counterpart to the persistence diagram in single-parameter persistence. The search for meaningful and tractable invariants remains an active and evolving frontier in this field. The \emph{multiparameter persistence landscape} \cite{Vipond2020} is one such invariant that also happens to be an extension of the classical persistence landscape. As such, the known functional and statistical properties of persistence landscapes carry over to the multiparameter setting \cite{Bubenik2015,Vipond2020}.

This paper extends the statistical framework of confidence bands from single-parameter \cite{ChazalFasyLecci2014a,ChazalFasyLecci2014} to multiparameter persistence landscapes. Our work is among the first to develop statistical methodology for inference with MPH. Our main contributions are: (i) a new \textit{functional central limit theorem} (CLT) for multiparameter persistence landscapes; (ii) \textit{confidence bands} for multiparameter persistence landscapes based on the CLT, and (iii) an algorithm to \textit{implement confidence bands on synthetic data}, illustrating their application in a classification task.

\section{Background}

We begin by providing the setting of our work and relevant background.

\subsection{Persistent Homology and Landscapes}
\label{sec:landscapes}

We overview persistent homology and its vectorization using multiparameter persistence landscapes.\\

\noindent
\textbf{Persistence Modules and Homology.} The central object in the theory of persistent homology is the \textit{persistence module}, a family of vector spaces $M_\bullet = \{M_{\xbf}: \xbf \in \Rbb^d\}$ indexed by $\Rbb^d$ with the product order, and linear maps $\varphi^{\xbf}_{\ybf}: M_{\xbf} \to M_{\ybf}$, for $\xbf \leq \ybf$, called \textit{internal} or \textit{transition maps}. For the remainder of this paper, we will assume that the parameters take values in a bounded box $B(T_1, \dots, T_d):=[0, T_1] \times \dots \times [0, T_d] \subset \Rbb^d$. Normalizing in each direction, we can simply consider the box $[0, T]^d$ with $T = \max\{T_i: 1\leq i \leq d\}$. 

Persistence modules arise in PH by taking the homology groups of the topological spaces in a filtration $K_\bullet = \{K_\xbf: \xbf \in \Rbb^d\}$ constructed from input data. For a topological space $K$, the \textit{$k$th homology group} $\H_k(K)$ contains information about the topological features of $K$: for $k=0$, these are its connected components; for $k=1$, its loops; for $k=2$, its bubbles and cavities, and so on for higher values of $k\geq 0$. The dimension of these vectors spaces, $\beta_k = \dim \H_k(K)$, corresponds to the number of these topological features.\\

\noindent
\textbf{Multiparameter Persistence Landscapes.} Given a persistence module $M_\bullet = \{M_{\xbf}: \xbf \in \Rbb^d\}$, its \textit{rank invariant} is defined from the ranks of the internal maps $\beta^{\xbf}_{\ybf} = \dim (\mathrm{Im}(\varphi^{\xbf}_{\ybf}))$ as the map $\rk: \Rbb^d\times \Rbb^d \to \Rbb$, such that
\[\rk(\xbf,\ybf) = \begin{cases}
    \beta^{\xbf}_{\ybf} & \text{if } \xbf\leq \ybf,\\
    0 & \text{otherwise}.
\end{cases}\]
The persistence landscape can be obtained from a re-scaling of this function in the following way.
\begin{definition}[Multiparameter Persistence Landscape \cite{Bubenik2015,Vipond2020}]
\label{def:landscape}
    Given a persistence module $M_\bullet =\{M_{\xbf}: \xbf \in \Rbb^d\}$, its persistence landscape $\lambda: \Nbb \times \Rbb^d \to \Rbb$ is defined as
    \[\lambda(k, \xbf) = \sup\{\epsilon>0: \beta^{\xbf-\hbf}_{\xbf+\hbf}\ge k \text{ for all } \hbf \ge \mathbf{0} \text{ with } \norm{\hbf}_\infty\leq \epsilon\}.\]
\end{definition}
In other words, the $k$th persistence landscape outputs the maximal radius (in the $\ell^\infty$ metric of $\Rbb^d)$ over which $k$ topological features persist in all positive directions for each $\xbf \in \Rbb^d$. From Lemma 20 in \cite{Vipond2020}, $\lambda(k, \xbf) \ge 0$ and it is a 1-Lipschitz function in $\xbf$ for all $k\ge 1$. In this paper, we focus on $\lambda(\xbf) := \lambda(1, \xbf)$, but the results hold for any fixed $k$, the crucial property being that $\lambda(\xbf)$ is 1-Lipschitz. We let $\Lcal$ be the space of landscapes $\lambda:[0,T]^d \to \Rbb$ for persistence modules with parameters in a bounded box $[0,T]^d \subset \Rbb^d$, which is a subset of the Banach space $L^p([0,T]^d)$, $1 \leq p \leq \infty$.\\ 

\subsection{Central Limit Theorems and Bootstrap Confidence Bands for Persistence Landscapes}

We now discuss the known convergence of single- and multiparameter landscapes as random variables (r.v.) in Banach space, given by central limit theorems. We also outline the construction of confidence bands---the dual construction to hypothesis testing---using bootstrap methods.\\

\noindent
\textbf{Central Limit Theorems (CLT)} hold for both single- and multiparameter persistence landscapes \cite{Bubenik2015,Vipond2020}: Suppose $X$ is a Borel measurable r.v.~on some probability space; 
then the persistence landscape for $X$, denoted $\Lambda = \Lambda(X)$, is a Banach space-valued r.v. Let $\{X_i\}$ be i.i.d.~copies of $X$ and $\Lambda_i$ be the corresponding landscapes, then given $\mathbb{E}[\|\Lambda\|]< \infty$ and $\mathbb{E}[\|\Lambda^2\|]< \infty$ we have that 
\(\sqrt{n}(\bar{\Lambda}_n - I_\Lambda)\),
where $I_\Lambda$ is the \textit{Pettis integral} of $\Lambda$, converges weakly to $G(\Lambda)$, a Gaussian r.v.~with the same covariance structure as $\Lambda$.

For the single-parameter case, a stronger functional CLT (also known as \textit{Donsker's Theorem}) exists \cite{ChazalFasyLecci2014}:
\[ \{\Gbb_n (\xbf)\}_{\xbf\in[0,T]} := \{\sqrt{n} (\overline{\Lambda}_n(\xbf) - I_\Lambda(\xbf))\}_{\xbf \in [0,T]}\]
converges weakly to a \emph{Gaussian process} on $[0,T]$. A Gaussian process is a family of real-valued r.v.s such that any finite collection from the family follows a multivariate normal distribution. The weak convergence to the limiting Gaussian process allows for the construction of confidence bands and hypothesis testing, which can be achieved using nonparametric methods such as bootstrap. The main contribution of this paper is to extend this framework to the multiparameter case in Section \ref{sec:fclt} and to demonstrate its utility in Section \ref{sec:simulations}.\\

\noindent
\textbf{Bootstrapping} 
is an inferential technique used to estimate population parameters from sample statistics by examining the relationship between the sample parameter and the statistic computed over a large number of resamples. In this paper, we focus on using the bootstrap to construct confidence bands for persistence landscapes. For a given functional parameter $\theta:\Omega \to \Rbb$, a \textit{confidence band} of level $\alpha$ for the parameter is given by a pair of functions $l,\,u: \Omega \to \Rbb$ such that $\Pbb(\theta(\omega) \in [l(\omega),u(\omega)],\, \forall \omega \in \Omega) \geq 1-\alpha$. Building on the single-parameter approach in \cite{ChazalFasyLecci2014,ChazalFasyLecci2014a}, we derive confidence bands for the mean multiparameter persistence landscape, as detailed in Algorithm \ref{alg:confidence_bands}. 

\section{A Functional Central Limit Theorem for Multiparameter Persistence Landscapes}\label{sec:fclt}

We extend the theory in \cite{ChazalFasyLecci2014a,ChazalFasyLecci2014} to prove a functional CLT for multiparameter persistence landscapes. Let $\Lambda_1, \ldots,  \Lambda_n$ be an i.i.d.~sample from a probability distribution $P$ on the space of multiparameter persistence landscapes $\Lcal$. 
Recall that we have the mean landscape  $I_\Lambda = \Ebb_{\Lambda \sim P} [\Lambda(\xbf)]$ and the sample average $\overline{\Lambda}_n (\xbf) = \frac{1}{n} \sum_{i=1}^n \Lambda_i(\xbf)$.

We consider the family of measurable functions $\Fcal = \{f_\xbf: \Lcal  \to \Rbb, \,  \lambda \mapsto f_\xbf(\lambda) = \lambda(\xbf)\}_{\xbf \in [0,T]^d}$ and define the empirical process indexed by this family:
\begin{equation}
\label{eq:empirical_process_landscapes}
    \{\Gbb_n f_\xbf\}_{f_\xbf\in \Fcal} := \{\Gbb_n (\xbf)\}_{\xbf \in [0,T]^d} = \{\sqrt{n} (\overline{\lambda}_n(\xbf) - \mu(\xbf))\}_{\xbf\in [0,T]^d},
\end{equation}
which we can rewrite as indexed by $[0,T]^d$ (the domain of the functions in $\Lcal$) for simplicity. 
A stochastic process $\{Y_n\}_{f \in \Fcal}$ indexed by $\Fcal$ \emph{converges weakly} to $Y$ in $\ell^\infty(\Fcal)$, denoted by $Y_n \rightsquigarrow Y$, if $\lim_{n \to \infty}\Ebb^*[g(Y_n)] = \Ebb[g(Y)]$ for every bounded, continuous function $g: \ell^\infty(\Fcal) \to \Rbb$; where $\Ebb^*$ denotes the outer expectation. Using the outer expectation here is a technical detail necessary in case the $Y_n$ are not Borel measurable, but it is not crucial for the remainder of the paper.
\begin{theorem}[Uniform Convergence of Multiparameter Persistence Landscapes]
\label{thm:fclt}
Let $\Gbb$ be a Gaussian process indexed by $\xbf \in [0,T]^d$ with mean zero and covariance function \(\kappa(\xbf, \ybf) = \int \lambda(\xbf) \lambda(\ybf) dP(\lambda) - \int \lambda(\xbf) dP(\lambda) \, \int \lambda(\ybf) dP(\lambda).\) Then $\Gbb_n \rightsquigarrow \Gbb$, where $\Gbb_n$ is the empirical process defined by an i.i.d.~sample of landscapes from \eqref{eq:empirical_process_landscapes}.
\end{theorem}

We need the following lemma from \cite{Kosorok2008}, Theorem 2.5, in order to prove Theorem \ref{thm:fclt}. This lemma provides a sufficient condition over the family of functions $\Fcal$ to ensure that the empirical process indexed by the family converges to a Gaussian process. Given $Q$ a probability on $\Lcal$, we use it to define a distance $L^2(Q)$ over the class of functions $\Fcal$ as $\norm{ f - g }_{Q, 2}^2 = \int |f - g |^2\, dQ$. Let $N(\Fcal, L_2(Q), \epsilon)$ be the covering number of $\Fcal$, i.e., the size of the smallest $\epsilon$-net in this metric. 
Lastly, recall that given a family of measurable functions $\Fcal =\{f: \Omega \to \Rbb\}$, a \textit{measurable envelope} of this collection is the ``smallest'' measurable function $F: \Omega \to \Rbb$ such that $f(\omega) \leq F(\omega)$ for almost every $\omega\in \Omega$. 

\begin{lemma}
\label{lemma:convergence}
    Let $\Fcal$ be a class of measurable functions satisfying 
    \[\int_0^1 \sqrt{\log \sup_Q N(\Fcal, L_2(Q), \epsilon \norm{F}_{Q,2})} \, d\epsilon < \infty\] 
    where $F$ is a measurable envelope of $\Fcal$ and the supremum is taken over all finitely discrete probability measures $Q$ with $\norm{F}_{Q,2}>0$. If $\Ebb[F^2] < \infty$, then $\Gbb_n \rightsquigarrow \Gbb$ with $\{\Gbb_n f\}_{f \in \Fcal}$ the empirical process indexed by $\Fcal$ and $\Gbb$ a Gaussian process with zero mean and covariance function \(\kappa(\xbf, \ybf) = \int \lambda(\xbf) \lambda(\ybf) dP(\lambda) - \int \lambda(\xbf) dP(\lambda) \, \int \lambda(\ybf) dP(\lambda).\)
\end{lemma}

\begin{proof}[Theorem \ref{thm:fclt}]
    Definition \ref{def:landscape} directly implies that the value of any landscape $\lambda \in \Lcal$ at any $\xbf \in [0,T]^d$ is bounded from above by the radius of biggest ball (in the $\ell^\infty$ metric of $\Rbb^d$) in $[0,T]^d$, which equals $T/2$. From this observation, we conclude that $F: \Lcal \to \Rbb$,  $F(\lambda) = T/2$ is a measurable envelope for $\Fcal = \{f_\xbf: \Lcal  \to \Rbb, \,  \lambda \mapsto \lambda(\xbf)\}_{\xbf \in [0,T]^d}$.

    Let $Q$ be a measure over $\Lcal$, and $\xbf,\, \ybf \in [0,T]^d$. The $L^2(Q)$ distance between $f_\xbf, \, f_\ybf \in \Fcal$ satisfies
    \begin{align*}
        \norm{f_\xbf - f_\ybf}^2_{Q,\, 2} & = \int_{\Lcal} |f_\xbf(\lambda) - f_\ybf(\lambda)|^2 \, dQ(\lambda)\\
        & = \int_{\Lcal} |\lambda(\xbf) - \lambda(\ybf)|^2 \, dQ(\lambda)\\
        & \leq \int_\Lcal \norm{\xbf - \ybf}_\infty^2\, dQ(\lambda) =  \norm{\xbf - \ybf}_\infty^2.
    \end{align*}

    We consider a partition $G$ denoted by $0 = t_0 < t_1 < \dots < t_N = T$ of the interval $[0,T]$ such that $\abs{t_i - t_{i+1}} = \epsilon \norm{F}_{Q, \, 2} =  \frac{\epsilon T}{2}$, where $F$ is the measurable envelope of $\Fcal$ as defined above.  Let $G^d$ be the corresponding grid obtained with this partition in $[0,T]^d$.
    We claim that $\{f_\tbf: \tbf \in G^d\}$ is an $(\epsilon T/2)$-net for $\Fcal$: For any $f_\xbf\in \Fcal$, there is some $\tbf = (t_{i_1}, \dots, t_{i_d}) \in G^d$ such that $t_{i_k} \leq x_k \leq t_{i_k +1}$ for $1 \leq k \leq d$. Then, we have
    \(\norm{f_\xbf - f_\tbf}_{Q, 2} \leq \norm{\xbf - \tbf}_\infty = \frac{\epsilon T}{2}\),
    proving that this family of functions is an $(\epsilon T/2)$-net.

    We can fit $2/\epsilon$ intervals of length $\epsilon T/ 2$ in $[0,T]$, meaning that $(2/\epsilon)^d$ is an upper bound for $N(\Fcal,\, L_2(Q),\, \epsilon \norm{F}_{Q,\, 2})$ for any probability measure $Q$, and thus for the supremum over $Q$. This means that
    \[\int_0^1 \sqrt{\log \sup_Q N(\Fcal, L_2(Q), \epsilon \norm{F}_{Q,\, 2})}\, d \epsilon  = \int_0^1 \sqrt{d(\log 2 - \log \epsilon)} \,d\epsilon < +\infty.\]
    Since the square of the measurable envelope is clearly integrable, using Lemma \ref{lemma:convergence}, we conclude that the empirical process indexed by $\Fcal$ in \eqref{eq:empirical_process_landscapes} converges to the Gaussian process in the statement of Theorem \ref{thm:fclt}. 
\end{proof}

\section{Confidence Bands for Persistence Landscapes: Method and Simulations}\label{sec:simulations}

We now give an algorithm for deriving confidence bands for multiparameter persistence landscapes; its implementation and the following experiments are available at: \url{https://github.com/inesgare/bands-mph-landscapes}.

\begin{algorithm}
    \caption{Bootstrap Confidence Bands Computation (Standard and Multiplier Bootstrap)}
    \label{alg:confidence_bands}
    \begin{algorithmic}[1]
        \Require Dataset or multiparameter landscapes $ \{\lambda_1, \lambda_2, \dots, \lambda_{n}\}$, number of bootstrap samples $B$, confidence level $\alpha$
        \Ensure Confidence bands $\ell_n, u_n: [0,T]^d \to \Rbb$
        
        \State Compute $\overline{\lambda}_n=n^{-1} \sum_{i=1}^n \lambda_i$
        
        \For{$b = 1$ to $B$}
            \If{using standard bootstrap}
                \State Draw a bootstrap sample $\{\lambda^\ast_1, \dots, \lambda_{\lfloor n/2\rfloor}^*\}$ by sampling with replacement
                \State Set $\hat{\theta}^*_b = \sup_{\xbf \in [0,T]^d}\abs{\sqrt{n} (\overline{\lambda}_{\lfloor n/2 \rfloor}^*(\xbf) - \overline{\lambda}_n(\xbf))}$
            \ElsIf{using multiplier bootstrap}
                \State Generate a set of i.i.d. multipliers $\xi_1, \dots \xi_ n \sim N(0,1)$
                \State Set $\hat{\theta}^*_b = \sup_{\xbf \in [0,T]^d} \frac{1}{\sqrt{n}}\abs{\sum_{i=1}^{n} \xi_i (\lambda_i^*(\xbf) - \overline{\lambda}_n(\xbf))}$
            \EndIf
        \EndFor
        
        \State Define $\tilde{Z}(\alpha) = \inf \left\{z : B^{-1}\sum_{b=1}^B I(\hat{\theta}^*_b > z) \leq \alpha \right\}$        
        \State Set $\ell_n(\xbf) = \overline{\lambda}_n(\xbf) - \frac{\tilde{Z}(\alpha)}{\sqrt{n}}$ and $u_n(\xbf) = \overline{\lambda}_n(\xbf) + \frac{\tilde{Z}(\alpha)}{\sqrt{n}}$
    \end{algorithmic}
\end{algorithm}

We demonstrate the confidence bands in practice through a use-case involving samples of \(N = 500\) points drawn from three distinct topological spaces: a sphere (radius \(R = 3\)), a torus (radii \(R = 3\), \(r = 0.7\)), and a Klein bottle, all embedded in \(\mathbb{R}^3 \). To introduce noise, we apply Gaussian noise with a scale of 0.1 and salt-and-pepper noise, where 0.5\% of the points are randomly displaced by up to 0.5 units. 
We compute the multiparameter landscapes using the \texttt{multipers} package \cite{multipers}, constructing 2D filtrations with one parameter being the scale (for the Vietoris--Rips filtration) and the other the value of a kernel density estimation (for the superlevel set filtration). We consider the first landscape of the 1-dimensional homology, that is, the landscape capturing the \textit{biggest} loop within the shape. 
Three instances of the input point clouds, along with their corresponding average landscapes and confidence bands with standard bootstrap (see Algorithm \ref{alg:confidence_bands}) for \(n = 100\) samples, are illustrated in Figure \ref{fig:confidence_bands}.

To demonstrate the utility of confidence bands, we additionally train a maximum depth band (MDB) functional classifier. The \emph{depth} of a landscape is defined by how often it lies within the confidence band for each shape. We classify landscapes based on the class that maximizes this depth. We perform classification with $n=100$ and $n=500$ landscapes per class, and report accuracies in Table \ref{tab:mbd-bootstrap-accuracies} after 5-fold validation.

\begin{figure}
    \centering
    \includegraphics[width=\linewidth]{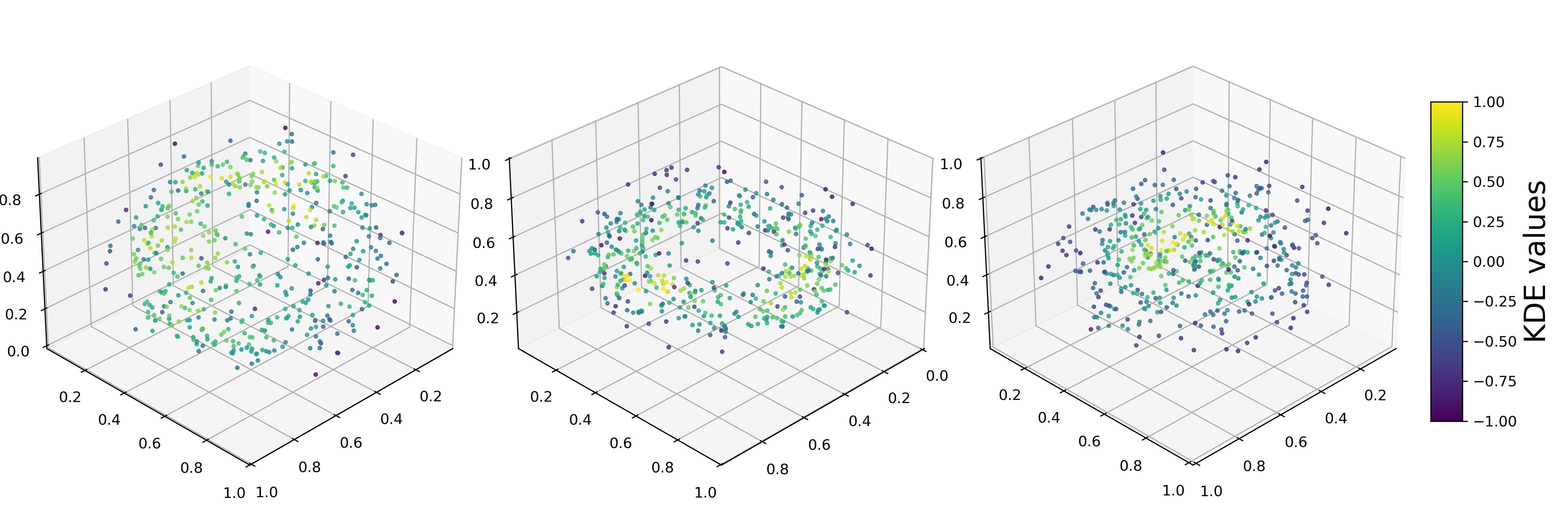}
    \includegraphics[width=\linewidth]{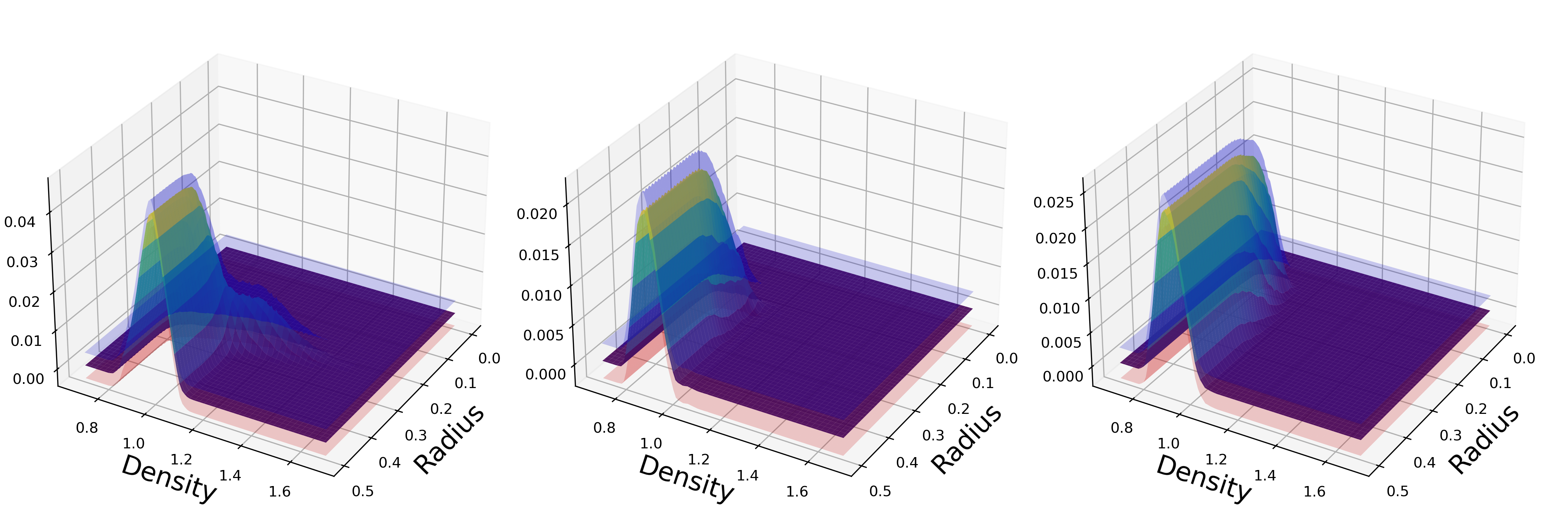}
    \caption{\textbf{Input point clouds and average landscapes with the standard bootstrap confidence bands.} \textit{Above:} Samples with $N=500$ points over a sphere of radius $R=3$ (left), a torus with radii $R=3$ and $r=0.7$ (center), and a Klein bottle (right) with Gaussian and salt and pepper noise added. Points are colored by the value of the Kernel density estimation. \textit{Below:} Average sample landscapes for $n=100$ samples and confidence bands with Standard Bootstrap.}
    \label{fig:confidence_bands}
\end{figure}

\begin{table}[]
\centering
\caption{\textbf{Mean accuracies after 5-fold cross validation} for the MBD classifier using the standard or multiplier bootstrap in single (SPH) or multiparameter (MPH) persistence. Models trained over $n$ subsamples of each class of shapes: spheres, torii and Klein bottles, as explained in Figure \ref{fig:confidence_bands}.\\}
\label{tab:mbd-bootstrap-accuracies}
\resizebox{\textwidth}{!}{%
\begin{tabular}{lcccc}
\hline
        &\, SPH Standard\,\,    & \,SPH Multiplier\,  & \, MPH Standard \,     & \, MPH Multiplier \,\  \\ \hline
$n=100$ & $0.97\pm 0.02$  & $0.93 \pm 0.03$ & $1.00\pm 0.00$    & $1.00 \pm 0.00$   \\
$n=500$ & $0.92 \pm 0.01$ & $0.88 \pm 0.02$ & $0.997 \pm 0.003$ & $0.985 \pm 0.004$ \\ \hline
\end{tabular}}
\end{table}

\section{Discussion}

This paper establishes the theoretical foundations and methodology for deriving confidence bands for multiparameter persistence landscapes, advancing statistical tools in this emerging subfield of TDA. Confidence bands hold intrinsic statistical value; we further demonstrate their practical utility in a classification task. Our results show that this method achieves near-perfect accuracy, with MPH outperforming SPH. However, accuracy slightly decreases as the number of subsamples increases, which result in narrower confidence bands.

The main limitation of our work is the computational bottleneck of computing multiparameter persistence landscapes, making our approach inherently constrained by the challenges of MPH. A key direction for future research is improving computational expense to then apply this methodology to real data.

\paragraph{Acknowledgements.} We wish to thank Jonathan Lau for interesting discussions that inspired this work. I.G.R.~is funded by a London School of Geometry and Number Theory--Imperial College London PhD studentship, which is supported by the UKRI EPSRC [EP/S021590/1]. A.M.~is supported by the UKRI EPSRC grant [EP/Y028872/1], Mathematical Foundations of Intelligence: An ``Erlangen Programme'' for AI. Q.W.~is funded by a CRUK--Imperial College London Convergence Science PhD studentship (2021 cohort, PIs Monod/Williams), which is supported by Cancer Research UK under grant reference [CANTAC721$\backslash$10021].

\bibliographystyle{alpha}
\bibliography{refs}

\end{document}